\documentclass[a4paper, 12pt]{amsart}
\usepackage[T2A]{fontenc}
\usepackage[utf8]{inputenc}
\usepackage[english, russian]{babel}
\usepackage{amsmath,amssymb,amsthm,url}
\usepackage{dsfont,textcomp,graphicx,overpic,wrapfig,pdflscape}
\usepackage{textcase}

\graphicspath{{figs/}}

\usepackage[usenames,dvipsnames]{color}

\makeatletter

\@addtoreset{figure}{section}
\makeatother

\usepackage{hyperref}
\hypersetup{
   colorlinks   = true, 
   urlcolor     = blue, 
   linkcolor    = blue, 
   citecolor   = red 
}

\def\R{{\mathbb R}} \def\Z{{\mathbb Z}}

\theoremstyle{plain}
\newtheorem{theorem}{Теорема}[section]
    \newtheorem{lemma}[theorem]{Лемма}
    
    \newtheorem{proposition}[theorem]{Утверждение}

\newtheorem{example}[theorem]{Пример}

\newtheoremstyle{mydefinition}
  {\medskipamount}
  {\medskipamount}
  {\normalfont}
  {\parindent}
  {\bfseries}
  {.}
  { }
  {}

\theoremstyle{mydefinition}

\newtheorem{remark}[theorem]{Замечание}

\begin{document}

\title{Число оборотов замкнутой ломаной вокруг точки}
 
\author{Э. Алкин, А. Мирошников, А. Скопенков}

\thanks{Заметка основана на материалах Летних конференций Турнира городов (\url{https://turgor.ru/lktg/index.php}) и курса <<Введение в топологию (дискретные структуры и алгоритмы в топологии)>> в Московском физико-техническом институте (проводимого вместе с И. Жильцовым и А. Руховичем, \url{https://old.mccme.ru//circles//oim/home/combtop13.htm\#fivt}). 
Благодарим  Е. Бордачеву (Дженжер) и О. Никитенко за совместную работу над теми материалами.} 

\address{\emph{Все авторы:} Московский физико-технический институт. 
\newline
\emph{А. Скопенков:} Независимый московский университет, \url{https://users.mccme.ru/skopenko/}.}

\date{}

\maketitle

\tableofcontents

\section{Введение}\label{s:intr}


\emph{Числом оборотов} замкнутой ориентированной ломаной $l$ на плоскости вокруг не лежащей на ней точки $O$ называется количество оборотов при вращении вектора, начало которого находится в точке $O$, а конец обходит ломаную $l$ в заданном направлении.
В этой заметке мы излагаем прямое элементарное строгое определение (\S\ref{s:wind}) и простейшие свойства (\S\ref{s:prop}, \S\ref{s:wint}) числа оборотов.   
Это определение проще, чем приводимое в некоторых учебниках. 
Например, в \cite{Wn} и в замечательной книге \cite{Ro15} приводятся определения числа оборотов через более сложные понятия поднятия / интеграла. 
При этом для обоснования существования поднятия / интеграла необходимо приведенное здесь элементарное определение. 

Мы покажем, как просто вычислять число оборотов: используя аддитивность (утверждение~\ref{k23}) или подсчет (знаков) точек пересечения (утверждения~\ref{p:stokes}.a и \ref{p:sign-stokes}.a; это дискретные версии важной \emph{теоремы Стокса} из математического анализа). 


На языке числа оборотов мы приведем элементарную формулировку маломерного случая  теоремы Борсука-Улама (теорема~\ref{borul-pl}). 
Эта теорема знаменита своими применениями в топологии, комбинаторике и математической экономике \cite{Ma03}. 

Число оборотов обобщается до \emph{степени отображения}, одного из основных инвариантов математики \cite{Ro15, Sk20}. 
О применениях числа оборотов см., например, \cite{Ro15, AS, AMS}.  
Подробнее о числе оборотов и связанных с ним понятиях см. \cite{Wn, Va81, To84, Ta88},  \cite[теорема 2]{KK18}, \cite[\S2]{Sk18}.  

Изучение большей части заметки не требует предварительных знаний (за пределами школьной программы), но требует 
математической грамотности. 

\section{Число оборотов: определение и обсуждение}\label{s:wind}

Далее все рассматриваемые точки, замкнутые ломаные и ломаные расположены на плоскости.
 
Пусть $O,A,B,A_1,\ldots,A_m$ "--- точки.  
 
Предположим, что $A\ne O$ и $B\ne O$ (но, возможно, $A=B$). 
\emph{Ориентированным (или направленным) углом} $\angle AOB$ называется число  $t\in(-\pi,\pi]$, такое что вектор $\overrightarrow{OB}$ сонаправлен вектору, полученному из $\overrightarrow{OA}$ вращением на угол $t$ против часовой стрелки. 
(Если Вы можете рассматривать векторы на плоскости как комплексные числа, то можете переписать это условие как $\overrightarrow{OB}\upuparrows e^{it}\overrightarrow{OA}$.)

Будем использовать без доказательства следующее утверждение (близкое к аксиомам): 

\begin{proposition}\label{p:angles} Для любых точек $A,B,C$ на плоскости и точки $O$, не лежащей на объединении отрезков $AB,BC,CA$,  

$\bullet$ $\angle OAB + \angle OBC + \angle OCA = \pm 2\pi$, если $O$ лежит в выпуклой оболочке точек $A,B,C$,

$\bullet$ $\angle OAB + \angle OBC + \angle OCA = 0$, иначе. 
\end{proposition} 

Заметим, что если в определении ориентированного угла взять $t\in[0,2\pi)$, то  аналогичное утверждение будет неверным. 

\textbf{Замкнутой ломаной} называется упорядоченный набор точек (не обязательно различных).\footnote{\label{fn:sets} Таким образом, замкнутая ломаная (определенная здесь) не является подмножеством плоскости. 
Тем не менее, иногда мы работаем с замкнутой ломаной $A_1\ldots A_m$ как с объединением отрезков $A_iA_{i+1}$, например, мы пишем <<ломаная, не проходящая через точку>>.}
В этом тексте мы иногда обозначаем упорядоченный набор $(A_1, \ldots, A_n)$ через $A_1\ldots A_n$ (допускается $n=1$).

Строго говоря, {\bf числом оборотов} замкнутой ломаной $l=A_1\ldots A_m$ вокруг не лежащей на ней точки $O$ называется 
$$w(l) = w(l,O) := \frac{\angle A_1OA_2+\angle A_2OA_3+\ldots+\angle A_{m-1}OA_m+\angle A_mOA_1}{2\pi}.$$

\begin{figure}[ht] \centering
    \includegraphics[scale=0.65]{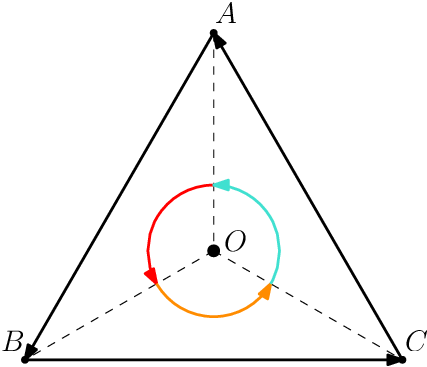}
    \qquad
    \includegraphics[scale=0.85]{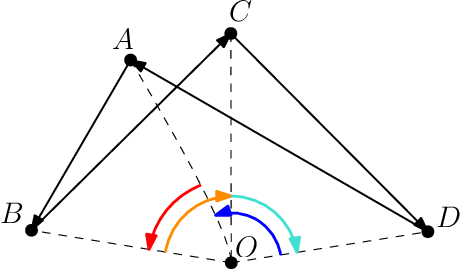}
    \caption{$w(ABC) = +1 $ и $w(ABCD)=  0$}
    \label{f:abco}
\end{figure}

Например, на рисунке~\ref{f:abco}
$w(ABC) = \dfrac{1}{2\pi} \left( \angle AOB + \angle BOC + \angle COA \right) = +1\quad\text{и}$
$$2\pi\cdot w(ABCD) = 
\angle AOB + \angle BOC + \angle COD + \angle DOA =  
\angle BOD + \angle DOB 
= 0.$$

\begin{figure}[ht] \centering
    \includegraphics[scale=1.0]{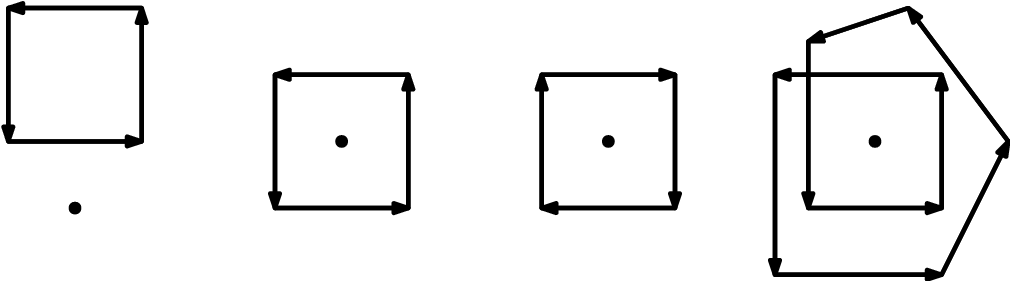}
    \caption{Числа оборотов равны $0,~+1,~-1,~+2$}
    \label{f:wn_examples}
\end{figure}

\begin{proposition}\label{p:winsim} Число оборотов 

(a) любого контура выпуклого многоугольника; 
    
(b) любой замкнутой ломаной без самопересечений

вокруг любой точки из внешности (внутренности) равно $0$ ($\pm 1$).
    См. рисунок~\ref{f:wn_examples}.
\end{proposition} 

\begin{proof}[Доказательство пункта (a)] 
    Обозначим через $\Omega$ данный выпуклый многоугольник, а через $\partial\Omega$ "--- его контур. 
    
    Если точка $O$ лежит вне $\Omega$, то проведeм через $O$ две опорные прямые к $\Omega$. 
    Возьмeм две точки $A, B$ из пересечений этих опорных прямых с контуром $\partial\Omega$ многоугольника. 
    Имеем $w(\partial \Omega) = \frac{1}{2\pi} (\angle AOB+\angle BOA) = 0$. 
    
    Если точка $O$ лежит внутри $\Omega$, то нарисуем правильный треугольник $ABC$ с центром в $O$.
    Возьмeм три точки $A', B', C'$ пересечения  лучей $OA, OB, OC$ с $\partial\Omega$.
    Они разбивают контур $\partial\Omega$ на три ломаные.
    Имеем
    $$w(\partial \Omega) = \frac{1}{2\pi} (\angle A'OB' + \angle B'OC' + \angle C'OA') = \frac{3}{2\pi} \angle A'OB' = \frac{3}{2\pi} \angle AOB = \pm 1.$$
\end{proof}

Пункт (b) в зависимости от изложения является либо следствием \emph{теоремы Жордана} \cite[\S1.4]{Sk20}, либо леммой в ее доказательстве.

Пусть $ABC$ "--- треугольник и  $O$ "--- точка внутри него.
Тогда $w(ABCABC) = 2 w(ABC) = \pm 2$.
Этот пример показывает, что числа оборотов для разных ломаных с одинаковым объединением их отрезков могут быть разными.
Более того, существует замкнутая ломаная $l$, такая что $w(l)=0$ для любой точки $O\in\R^2-l$.
Примерами являются ломаные $A$, $AB$, $ABCB$ для любых точек $A$, $B$ и $C$.

\begin{example}\label{p:ex} Для любых целого числа $n$ и точки $O$ существует замкнутая ломаная, число оборотов которой вокруг $O$ равно $n$.
\end{example}

\begin{proof}[Построение]
    Если $n=0$, то примером является замкнутая ломаная, состоящая из одной точки.
    Если $n\ne 0$, то возьмем правильный треугольник $ABC$ с центром $O$, ориентированный против часовой стрелки при $n>0$ и по часовой стрелке иначе.
    Определим замкнутую ломаную $l$ формулой $l := \underbrace{ABC\ldots ABC}_{|n| \text{ раз}}$.
    По утверждению~\ref{p:winsim}.a, имеем $w(l) = |n| \cdot w(ABC) = n$.
\end{proof}

\begin{proposition}\label{p:noncl} 
    Число оборотов $w(A_1\ldots A_m)$ является целым числом. 
\end{proposition}

\begin{proof}
    Для $j=1,\ldots,m$ обозначим $t_j:=\angle A_jOA_{j+1}$, где $A_{m+1}=A_1$.  
    Тогда 
    $$\overrightarrow{OA_1} \upuparrows e^{it_m}\overrightarrow{OA_m} \upuparrows e^{i(t_m+t_{m-1})}\overrightarrow{OA_{m-1}} \upuparrows \ldots \upuparrows e^{i(t_m+t_{m-1}+\cdots+t_1)}\overrightarrow{OA_1}.$$
    Следовательно, $(t_m+t_{m-1}+\cdots+t_1)/2\pi$ является целым числом.
\end{proof}

\emph{Указание к альтернативному доказательству:} 
ввиду утверждения~\ref{p:angles},
$$\angle A_{m-1}OA_m+\angle A_mOA_1 \equiv \angle A_{m-1}OA_1 \mod2\pi.$$
\emph{Еще одно доказательство} утверждения~\ref{p:noncl} дают утверждения~\ref{p:winsim}.a и \ref{p:recur}. 

\begin{remark}[обобщения]
(a) Определим \textit{число оборотов} $w(l,O)$ для замкнутой \emph{непрерывной} кривой $l$. 
Последняя определяется как непрерывное отображение $l:S^1\to\R^2-\{O\}$ окружности в плоскость без точки.
Так как отображение $l$ равномерно непрерывно, то существует целое $m>0$, такое что  для любых разбиения окружности $S^1$ на $m$~равных дуг, и~точек $x,y$ одной дуги, точки $l(x)$ и $l(y)$ не симметричны относительно $O$.
Для $k=1,\ldots,m$ обозначим через $A_k$ конец $k$-й дуги.
Определим 
$$2 \pi \cdot w(l,O) := 
\angle A_1OA_2+\angle A_2OA_3+\ldots+\angle A_{m-1}OA_m+\angle A_mOA_1.$$
Это определение корректно (см., например, \cite[\S3]{Sk20}). 

(b) \textit{Числом оборотов} $w(C, O)$ целочисленного 1-цикла \cite[\S4]{Sk} $C$ вокруг точки $O$, которая не принадлежит ни одному отрезку из $C$, назовем число
$$w(C, O) = \sum_{AB \in C} \angle AOB.$$
Аналогично утверждению~\ref{p:noncl}, число $w(C, O)$ является целым.
Какие ещё утверждения обобщаются на 1-циклы?
\end{remark}
 

\section{Простейшие свойства числа оборотов}\label{s:prop} 
 
\begin{theorem}[Борсук--Улам]\label{borul-pl} 
    Пусть замкнутая ломаная $A_1\ldots A_{2k}$ не проходит через точку $O$ и симметрична относительно $O$ (т.~е. $O$ "--- середина отрезка $A_jA_{k+j}$ для каждого $j=1,\ldots,k$).
    Тогда число оборотов ломаной вокруг точки $O$ нечетно.
\end{theorem}

\begin{figure}[ht] \centering
    \includegraphics[scale=0.65]{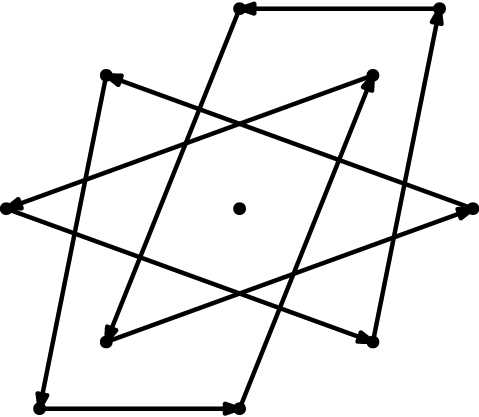}
    \caption{Замкнутая ломаная $l$, симметричная относительно точки $O$; $w(l)=3$} 
\end{figure}

Следующие обозначение и результаты будут полезны.

\textbf{Ломаной} называется упорядоченный набор точек (не обязательно различных).\footnote{В математике иногда разные вещи 
имеют одинаковые формализации --- тогда разница проявляется в том, что мы с этими вещами делаем.
Например, многие разные вещи формализуются понятием <<множества>>.
Тогда формально одинаковые объекты называются по-разному, чтобы сделать более понятными операции над ними.
Например, в этом тексте, упорядоченный набор точек называется и ломаной, и замкнутой ломаной. 
Другой пример: чтобы сделать более понятной операцию умножения, <<числовую последовательность>> называют <<производящей функцией этой последовательности>>.
}

Пусть $l=A_1\ldots A_m$ "--- ломаная, не проходящая через точку $O$. 
Определим действительное число $w'(l)=w'(l,O)$ формулой
$$2\pi \cdot w'(l) = 2\pi \cdot w'(l,O):= \angle A_1OA_2+\angle A_2OA_3+\ldots+\angle A_{m-1}OA_m.$$
Очевидно, что

$\bullet$ $w'(A_1\ldots A_mA_1) = w(A_1\ldots A_m)$;

$\bullet$ $w'(A_1\ldots A_m) = w'(A_1\ldots A_j) + w'(A_j\ldots A_m)$ для каждого $j=1,\ldots,m$; 

$\bullet$ если точки $A_2,\ldots, A_{m-1}$ лежат внутри угла $\angle A_1 O A_m$, то  
$2\pi w'(A_1\ldots A_m) = \angle A_1 O A_m$.

\begin{proposition}\label{p:w'}
    Имеем $\angle A_1OA_m = 2\pi w'(A_1\ldots A_m)+2\pi n$ для некоторого целого $n$. 
\end{proposition}

\begin{proof}
    Очевидно, что $2\pi w(A_1\ldots A_m) = 2\pi w'(A_1\ldots A_m)+\angle A_mOA_1$.
    По утверждению~\ref{p:noncl} получаем требуемое.
\end{proof}

\begin{proof} [Доказательство теоремы~\ref{borul-pl}]
    Ввиду симметричности, $w'(A_1 \ldots A_{k+1}) = w'(A_{k+1} \ldots A_{2k}A_1)$. 
    Тогда
    $$w(A_1\ldots A_{2k}) = w'(A_1\ldots A_{2k}A_1) = w'(A_1\ldots A_{k+1}) + w'(A_{k+1}\ldots A_{2k}A_1) =$$
    $$= 2w'(A_1 \ldots A_{k+1}) \stackrel{(1)}{=} 2 \left( \frac{\angle A_1OA_{k+1}}{2\pi} + n \right) = 1 + 2n$$
    для некоторого целого $n$. 
    Здесь равенство $(1)$ следует из утверждения~\ref{p:w'}.
\end{proof}

Обозначим через $l^{-1}$ ломаную, полученную из ломаной $l$ прохождением в противоположном порядке.
Очевидно, что $w'(l) = -w'(l^{-1})$ для любой точки $O\not\in l$. 

\emph{Конкатенацией} ломаных $l_1 = A_1 \ldots A_m C$ и $l_2 = C B_1\ldots B_k$ называется  ломаная 
$$l_1l_2 := A_1 \ldots A_m C B_1 \ldots B_k.$$
\emph{Конкатенацией} замкнутых ломаных $l_1 = A_1 \ldots A_m C$ и $l_2 = B_1\ldots B_k C$ называется замкнутая ломаная 
$$l_1l_2 := A_1 \ldots A_m C B_1 \ldots B_k C.$$
В этих обозначениях 
$$w'(l_1l_2) = w'(l_1) + w'(l_2)$$ 
для любых точки $O$ и ломаных $l_1, l_2$, не проходящих через $O$.
Это очевидное равенство уже было использовано в доказательстве теоремы~\ref{borul-pl}. 

\begin{proposition}\label{k23}
    Пусть $O$, $A$, $B$ "--- три попарно различные точки.
    Пусть $l_1$, $l_2$, $l_3$ "--- ломаные, соединяющие точку $A$ с точкой $B$, и не проходящие через точку $O$.
    Тогда 
    \linebreak
    $w(l_1 l_2^{-1}) + w(l_2 l_3^{-1}) = w(l_1 l_3^{-1})$.
\end{proposition}

\begin{proof}
    Имеем
    $$w(l_1 l_2^{-1}) + w(l_2 l_3^{-1}) = w'(l_1) + w'(l_2^{-1}) + w'(l_2) + w'(l_3^{-1}) = w'(l_1) + w'(l_3^{-1}) = w(l_1 l_3^{-1}).$$  
\end{proof}

\begin{example}[см. построение в {\cite[\S2.A]{AMS}}] \label{k23-ex}
    Пусть $O$, $A$, $B$ "--- три попарно различных точки.
    Для любых трех целых чисел $n_1, n_2, n_3$, таких что $n_1 + n_2 = n_3$, существуют три ломаные $l_1, l_2, l_3$, соединяющие точку $A$ с точкой $B$, не проходящие через точку $O$, и такие, что
$$w(l_1l_2^{-1}) = n_1, \qquad w(l_2l_3^{-1}) = n_2 \quad\text{и}\quad w(l_1l_3^{-1}) = n_3.$$
\end{example}

\begin{proposition} \label{p:recur}
    Пусть замкнутая ломаная $A_1 \ldots A_m$ не проходит через точку $O$.
Пусть точка $P$ такова, что $O$ не лежит ни на одном из отрезков $PA_1,\ldots,PA_m$.
    Тогда
$$w(A_1 \ldots A_m) = \sum_{i = 1}^{m}w(PA_iA_{i+1}),\quad\text{где}\quad A_{m+1} = A_1.$$
\end{proposition}

\begin{proof}
    Обозначим $a_i := A_iA_{i+1}$ и $p_i = PA_i$ для каждого $i = 1, \ldots, m$.
    Тогда многократным применением утверждения~\ref{k23} получаем 
    $$w(A_1 \ldots A_m) = w(a_1\ldots a_m) =  \sum_{i = 1}^{m} w(p_{i-1}a_ip_i^{-1}) = \sum_{i = 1}^{m}w(PA_iA_{i+1}),\quad\text{где}\quad p_0 = p_m.$$
\end{proof}

\begin{proposition}\label{p:3rays} 
    Пусть $A_0A_1A_2$ "--- правильный треугольник, а точка $O$ "--- его центр.
    Для $j=0,1,2$ обозначим через $l_j$ ломаную, не пересекающую луч $OA_j$ и соединяющую $A_{j+1}$ с $A_{j+2}$, где нумерация берется по модулю 3.
    Тогда $w(l_0l_1l_2) = \pm 1$.  
\end{proposition}

\begin{proof}
    Можно считать, что вершины треугольника пронумерованы против часовой стрелки.
    По утверждению~\ref{p:w'} имеем $w'(l_0) = \frac{1}{3} + k$ для некоторого целого $k$.
    В следующем абзаце мы докажем, что $k = 0$.
    Равенства  $w'(l_1) = w'(l_2) = \frac{1}{3}$ доказываются аналогично.
    Тогда $w(l_0l_1l_2) = w'(l_0) + w'(l_1) + w'(l_2) = 1$. 

    Обозначим через $B_1 \ldots B_m$ последовательные вершины ломаной $l_0$. 
    Для $j = 1, \ldots, m-1$ положим $t_j := \angle B_jOB_{j+1} \in (-\pi, \pi)$.
Если бы нашлось  $j= 1, \ldots, m-1$, такое что $T_j := t_1 + \ldots + t_j \not \in \left(-\frac{2\pi}{3}, \frac{4\pi}{3}\right)$, то для наименьшего из таких $j$ отрезок $B_jB_{j+1}$ пересекал луч $OA_0$. 
Значит, $T_j \in \left(-\frac{2\pi}{3}, \frac{4\pi}{3}\right)$ для любого $j = 1, \ldots, m-1$.
    В частности, $T_{m-1} \in \left(-\frac{2\pi}{3}, \frac{4\pi}{3}\right)$.
    При этом $T_{m-1} = 2\pi w'(l_0) = \frac{2\pi}{3} + 2\pi k$ для некоторого целого $k$.
    Следовательно, $k = 0$.
\end{proof}

Другие доказательства приведены в \S\ref{s:wint} и в \cite[конец \S3]{AS}.

\section{Число оборотов и 
пересечения}\label{s:wint} 

\begin{proposition} \label{p:rel} 
Возьмем точки $P_0$ и $P_1$, соединенные ломаной, не пересекающейся с замкнутой ломаной $l$. 
Тогда  $w(l,P_0)=w(l,P_1)$. 
\end{proposition}
    
\emph{Указание:} используйте \emph{соображения непрерывности} или выведите из леммы~\ref{l:square}.a.

\medskip
Несколько точек находятся в \emph{общем положении}, если никакие три из них не лежат на прямой и никакие три отрезка, их соединяющие, не имеют общей внутренней точки.
\emph{Далее в этом разделе вершины всех рассматриваемых ломаных $l$, а также всех пар ломаных $l$ и $p$, попарно различны и (если не оговорено противное) находятся в общем положении.}

\begin{proposition}\label{p:stokes} 
    Возьмем замкнутую и незамкнутую ломаные $l$ и $p$.
    Обозначим через $P_0$ и $P_1$ начальную и конечную точки ломаной $p$.
    Тогда
    
(a) $w(l,P_1)-w(l,P_0) \equiv |l\cap p| \mod2$.

(b) $w(l, P_1) \equiv |l\cap p| \mod2$, если точка $P_0$ расположена достаточно далеко от ломаной $l$ (например, вне выпуклой оболочки ломаной $l$).
\end{proposition}

\begin{proof}[Набросок доказательства]
    Пункт (b) следует из пункта (a), так как $w(l, P_0) = 0$.

(a) Если $p$ не пересекает $l$, то по утверждению~\ref{p:rel} имеем $w(l,P_1)-w(l, P_0) = 0 = |l \cap p|$.
     
Иначе можно считать, что $p$ есть отрезок, пересекающий $l = A_1 A_2 \ldots A_m$ в единственной точке, и эта точка принадлежит отрезку $A_1A_2$.
    Выберем точку $X$ так, что $P_0$ находится внутри треугольника $A_1XA_2$.  
    Тогда $P_1$ лежит вне треугольника $A_1XA_2$.
    Определим замкнутую ломаную $l'$ формулой $l' := A_1 X A_2 \ldots A_m$.
    Имеем 
    $$w(l, P_1) = w(l, P_1) + w(A_1XA_2, P_1) = w(l', P_1) \stackrel{(1)}{=} 
w(l', P_0) = w(l, P_0) + w(A_1XA_2, P_0) = w(l, P_0) \pm1.$$
    Здесь равенство (1) выполнено по утверждению~\ref{p:rel}.

\emph{Другой способ:} можно применить лемму~\ref{l:square}.b (или даже ее версию по модулю 2). 
\end{proof}

\begin{figure}[ht]\centering
    \includegraphics[scale=.8]{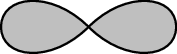} 
    \qquad 
    \includegraphics[scale=.8]{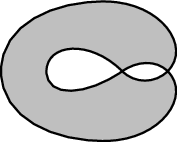}
    \qquad
    \includegraphics[scale=.8]{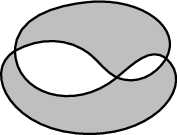}
    \qquad 
    \includegraphics[scale=.8]{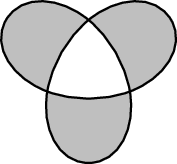}
    \qquad 
    \includegraphics[scale=.8]{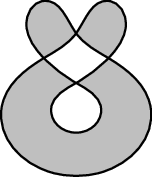}
    \caption{Шахматные раскраски дополнений и внутренности по модулю 2}
    \label{f:tohu}
\end{figure}

Любое (открытое) подмножество плоскости разбивается на \emph{компоненты (связности)}, такие что любые две точки одной компоненты можно соединить ломаной, лежащей в подмножестве, а никакие две точки из разных компонент --- нельзя.
Дополнение $\R^2-l$ до замкнутой ломаной $l$ допускает <<шахматную раскраску>>, такую что компоненты дополнения, соседствующие по некоторому отрезку ломаной, покрашены в разные цвета, см. рисунок~\ref{f:tohu} \cite[утверждение 2.1.1]{Sk18}.  
\textit{Внутренностью по модулю 2} замкнутой ломаной $l$ называется объединение черных областей шахматной раскраски (при условии, что <<бесконечная>> область белая).
Вот эквивалентное прямое определение, не использующее шахматной раскраски. 
\textit{Внутренностью по модулю 2} называется множество всех точек $X\in\R^2-l$, для которых найдется ломаная $p$, 

$\bullet$ соединяющая точку $X$ с точкой, расположенной достаточно далеко от $l$, 

$\bullet$ пересекающая $l$ в нечетном числе точек, и 

$\bullet$ такая, что все вершины ломаных $l$ и $p$ попарно различны и (кроме $X$) находятся в общем положении. 

Корректность этого определения следует из леммы о четности \cite[лемма 1.3.2.b]{Sk18}: \emph{если вершины двух замкнутых ломаных попарно различны и находятся в общем положении, то ломаные пересекаются в четном числе точек}. 
(Эта лемма также есть частный случай утверждения~\ref{p:stokes}.a для $P_0 = P_1$.)

\begin{proposition}[следует из утверждения~\ref{p:stokes}.b]\label{p:radon}
Точка $O$ лежит во внутренности по модулю 2 замкнутой ломаной $l$ тогда и только тогда, когда $w(l)$ нечетно.
\end{proposition}

\begin{figure}[h]\centering
\includegraphics[scale=.8]{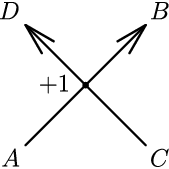}\quad \includegraphics[scale=.8]{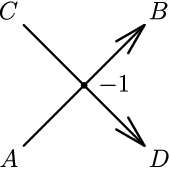}
\caption{Знак точки пересечения}
\label{f:sign}
\end{figure}

\begin{figure}[h]
\centerline{\includegraphics[width=4.5cm]{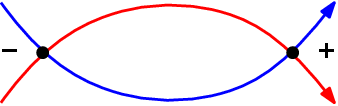} \qquad \includegraphics{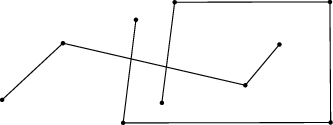}}
\caption{Две ломаные на плоскости, пересекающиеся в четном числе точек, сумма знаков которых равна нулю (слева) и не равна нулю (справа)} 
\label{f:gl2}
\end{figure}

Пусть никакие три из точек $A,B,C,D$ не лежат на одной прямой.
{\bf Знаком} точки пересечения ориентированных отрезков $\overrightarrow{AB}$ и $\overrightarrow{CD}$ назовем $+1$, если обход $ABC$ происходит по  часовой стрелке, и $-1$ в противном случае (рисунок~\ref{f:sign}).
Для ломаных $l$ и $p$ обозначим через $l \cdot p$ сумму знаков точек их пересечений (рисунок~\ref{f:gl2}).

\begin{proposition}\label{p:sign-stokes} 
    Возьмем замкнутую и незамкнутую ломаные $l$ и $p$.
    Обозначим через $P_0$ и $P_1$ начальную и конечную точки ломаной $p$.
    Тогда
    
(a) $w(l,P_1)-w(l,P_0) = l \cdot p$.  

(b) $w(l,P_1) = l \cdot p$, если точка $P_0$ расположена достаточно далеко от ломаной $l$. 
\end{proposition}

Утверждение~\ref{p:sign-stokes} доказывается аналогично утверждению~\ref{p:stokes} или выводится из леммы~\ref{l:square}.b. 

\begin{figure}[ht]\centering
    \includegraphics[scale=.8]{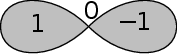}
    \qquad
    \includegraphics[scale=.8]{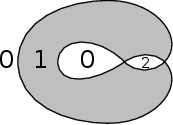}
    \qquad
    \includegraphics[scale=.8]{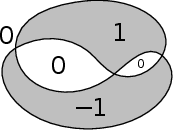}
    \qquad
    \includegraphics[scale=.8]{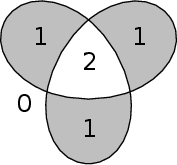}
    \qquad
    \includegraphics[scale=.8]{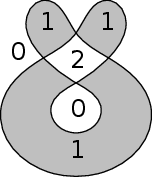}
    \caption{<<Раскраски>> дополнений в соответствии с числом оборотов (с точностью до смены направления ломаной)}
    \label{f:tohu_colored}
\end{figure}

Возьмем замкнутую ломаную $l$ (рисунок~\ref{f:tohu_colored}).
По утверждению~\ref{p:rel}, числа оборотов вокруг 
точек из одной компоненты дополнения $\R^2-l$ одинаковы. 
По утверждению~\ref{p:sign-stokes}.a, числа оборотов в компонентах, соседствующих по некоторому отрезку, отличаются на $\pm 1$.
Такая <<раскраска>> дополнения $\R^2-l$ целыми числами называется \emph{раскраской Мебиуса-Александера}. 

Для ломаных $A\ldots B$ и $C\ldots D$ обозначим 
$$\partial(A\ldots B \times C\ldots D) := w'(A\ldots B,D)-w'(C\ldots D,B)-w'(A\ldots B,C)+w'(C\ldots D,A).$$ 
Это число определено, когда ни один из концов каждой ломаной не лежит на объединении звеньев другой ломаной. 
(Представим ломаные в качестве \emph{кусочно-линейных отображений} $l,p:[0,1]\to\R^2$.
Тогда $\partial(l\times p)$ есть число оборотов вектора $l(x)-p(y)$ при обходе точкой $(x,y)$ границы $\partial([0,1]^2)$ квадрата $[0,1]^2$ против часовой стрелки.)

Для ломаных $l,p$ на рисунке~\ref{f:gl2} слева и справа $\partial(l\times p)$ равно $0$ и $\pm2$, соответственно (проверьте!). 
 

\begin{lemma}\label{l:square} 
(a) Для любых непересекающихся ломаных $l$ и $p$ 
выполнено $\partial(l\times p)=0$. 

(b) Для любых ломаных $l$ и $p$ выполнено $\partial(l\times p) = l\cdot p$.  
\end{lemma}

\begin{proof}[Набросок доказательства]
(a) Если $l=AB$ и $p=CD$, то лемма верна, поскольку 
\linebreak
$\angle ADB + \angle DBC + \angle BCA + \angle CAD = 0$ для любых непересекающихся отрезков $AB$ и $CD$.
(Достаточно, чтобы эта сумма углов лежала в $(-2\pi,2\pi)$.)

Сведем общий случай к рассмотренному.     
Для произвольных ломаных $l_1, l_2, p$ имеем 
    $$\partial(l_1l_2 \times p) = \partial(l_1 \times p) + \partial(l_2 \times p)\quad\text{и} \quad\partial(p \times l_1l_2) = \partial(p \times l_1) + \partial(p \times l_2).$$
    Поэтому 
    $$\partial(l \times p) = \sum_{\begin{array}{l}
         EF  \text{ --- отрезок ломаной } l  \\
         GH  \text{ --- отрезок ломаной } p    
         \end{array}} \partial(EF\times GH)=0.$$

 (b) Аналогично п. (a). 
\end{proof}

\begin{proof}[Доказательство утверждения~\ref{p:3rays} при помощи утверждения~\ref{p:rel}] 
(Получено редактированием текста А. Абзалилова.)
Определим ломаную $a_j := A_{j+1}A_{j+2}$ и замкнутую ломаную $m_j := l_ja_j^{-1}$.
     Тогда 
     $$w(l_0l_1l_2) \stackrel{(1)}{=} w(m_0) + w(m_1) + w(m_2) + w(a_0a_1a_2) \stackrel{(2)}{=} w(A_0A_1A_2) = \pm1\quad\text{где}$$

$\bullet$ равенство (1) есть многократное применение утверждения~\ref{k23}; 

$\bullet$ равенство (2) следует из равенств $w(m_0) = w(m_1) = w(m_2) = 0$, доказанных в следующем абзаце.

Возьмем любое 
$j = 0, 1, 2$.
     На луче $OA_j$ найдется точка $P_j$ вне выпуклой оболочки точек ломаной $m_j$.
     Тогда $w(m_j, P_j) = 0$.
     Отрезок $OP_j$ не пересекается с ломаной $m_j$.
     Следовательно, $w(m_j) = 0$ по утверждению~\ref{p:rel}.
\end{proof}

 
\end{document}


\title{The Winding Number of a Closed Curve 
Around a Point}
 
\author{E. Alkin, A. Miroshnikov, A. Skopenkov}

\thanks{This note is based on materials from the Summer Conferences of the Tournament of Towns (\url{https://turgor.ru/lktg/index.php}) and the course `Introduction to Topology (Discrete Structures and Algorithms in Topology)' at Moscow Institute of Physics and Technology (taught jointly with I. Zhiltsov and A. Rukhovich, \url{https://old.mccme.ru//circles//oim/home/combtop13.htm\#fivt}). 
We acknowledge E. Boracheva (Dzhenzher) and O. Nikitenko for useful discussions.}

\address{\emph{All authors:} Moscow Institute of Physics and Technology. 
\newline
\emph{A. Skopenkov:} Independent University of Moscow, \url{https://users.mccme.ru/skopenko/}.}

\date{}
\maketitle

\begin{abstract} 
In this expository note we present an elementary direct rigorous definition and the simplest properties of the winding number. 
This definition is simpler than the one given in some textbooks. 
We show how to compute the winding number easily: using additivity or counting the (signed) intersection points. 
In the language of the winding number, we present an elementary formulation and proof of the low-dimensional case of the Borsuk--Ulam theorem. 
An English version is followed by a Russian version. 
\end{abstract} 



\tableofcontents

\section{Introduction}\label{s:intr}

The \emph{winding number} of a closed oriented polygonal line $l$ in the plane around a point $O$ not lying on $l$ is the number of turns made by the vector whose initial point is $O$, and whose endpoint traverses $l$ in the given direction.
In this note we expose an elementary direct rigorous definition (\S\ref{s:wind}) and the simplest properties (\S\ref{s:prop}, \S\ref{s:wint}) of the winding number.   
This definition is simpler than the one given in some textbooks. 
For example, in \cite{Wn} and in the excellent book \cite{Ro15}, definitions of the winding number are given using more complicated notions of lifting / integral. 
Thereby the justification of the existence of the lifting / integral requires the elementary definition given in this note. 

We show how to compute the winding number easily: using additivity (Proposition~\ref{k23}) or counting the (signed) intersection points (Propositions \ref{p:stokes}.a and \ref{p:sign-stokes}.a; these are discrete versions of the important \emph{Stokes theorem} from mathematical analysis). 

In the language of the winding number, we present an elementary formulation and proof of the low-dimensional case of the Borsuk--Ulam theorem (Theorem~\ref{borul-pl}). 
This theorem is famous for its applications in topology, combinatorics, and mathematical economics \cite{Ma03}. 

The winding number generalizes to the \emph{degree of a map}, one of the basic invariants of mathematics \cite{Ro15, Sk20}. 
For applications of the winding number, see for example \cite{Ro15, AS, AMS}.  
For more on the winding number and related notions, see \cite{Wn, Va81, To84, Ta88}, \cite[Theorem 2]{KK18}, \cite[\S2]{Sk18}.  

Studying most of this note requires no preliminary knowledge (beyond the high-school curriculum), but does require some mathematical literacy. 

\section{The winding number: definition and discussion}\label{s:wind}

In the rest of this text, all points, closed polygonal lines, and polygonal lines are considered in the plane.

Let $O,A,B,A_1,\ldots,A_m$ be points.  
 
Assume that $A\ne O$ and $B\ne O$ (but possibly $A=B$). 
The \emph{oriented (a.k.a. directed) angle} $\angle AOB$ is the number $t\in(-\pi,\pi]$ such that the vector $\overrightarrow{OB}$ is codirected with the vector obtained from $\overrightarrow{OA}$ by rotation through angle $t$ counterclockwise. 
(If you can regard vectors in the plane as complex numbers, then you can rewrite this condition as $\overrightarrow{OB}\upuparrows e^{it}\overrightarrow{OA}$.)

We shall use without proof the following statement  (close to an axiom): 

\begin{proposition}\label{p:angles} For any points $A,B,C$ in the plane and any point $O$ not lying on the union of the segments $AB,BC,CA$,  

$\bullet$ $\angle OAB + \angle OBC + \angle OCA = \pm 2\pi$, if $O$ lies in the convex hull of the points $A,B,C$,

$\bullet$ $\angle OAB + \angle OBC + \angle OCA = 0$, otherwise. 
\end{proposition} 

Note that if in the definition of the oriented angle one takes $t\in[0,2\pi)$, then the analogous statement becomes false. 

A \textbf{closed polygonal line} is an ordered collection of points (not necessarily distinct).\footnote{\label{fn:sets} Thus, a closed polygonal line (as defined here) is not a subset of the plane. 
Nevertheless, sometimes we work with a closed polygonal line $A_1\ldots A_m$ as with the union of the segments $A_iA_{i+1}$; for example, we write `a polygonal line not passing through a point'.}
In this text we sometimes denote the ordered collection $(A_1, \ldots, A_n)$ by $A_1\ldots A_n$ ($n=1$ is allowed).

Let $l = A_1\ldots A_m$ be a closed polygonal line not passing through a point $O$. 
Rigorously, the {\bf winding number} of $l$ around $O$ is defined by
$$w(l) = w(l,O) := \frac{\angle A_1OA_2+\angle A_2OA_3+\ldots+\angle A_{m-1}OA_m+\angle A_mOA_1}{2\pi}.$$

\begin{figure}[ht] \centering
    \includegraphics[scale=0.65]{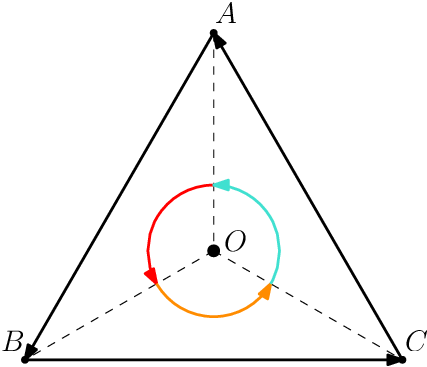}
    \qquad
    \includegraphics[scale=0.85]{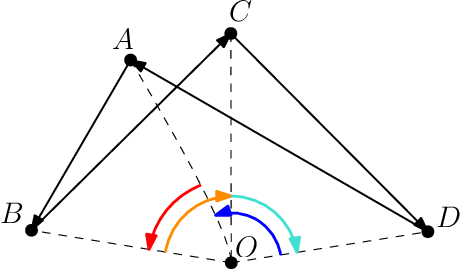}
    \caption{$w(ABC) = +1 $ and $w(ABCD)=  0$}
    \label{f:abco}
\end{figure}

For example, in Figure~\ref{f:abco}
$w(ABC) = \dfrac{1}{2\pi} \left( \angle AOB + \angle BOC + \angle COA \right) = +1\quad\text{and}$
$$2\pi\cdot w(ABCD) = 
\angle AOB + \angle BOC + \angle COD + \angle DOA =  
\angle BOD + \angle DOB 
= 0.$$

\begin{figure}[ht] \centering
    \includegraphics[scale=1.0]{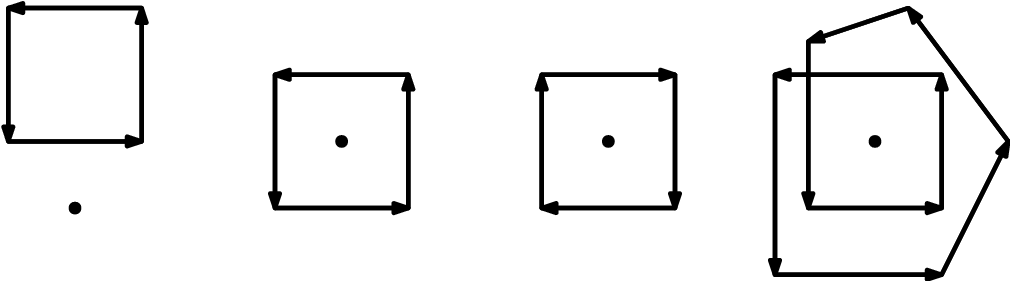}
    \caption{The winding numbers are $0,~+1,~-1,~+2$}
    \label{f:wn_examples}
\end{figure}

\begin{proposition}\label{p:winsim} The winding number of 

(a) the boundary of any convex polygon; 
    
(b) any simple closed polygonal line

around any point in the exterior (interior) equals $0$ ($\pm 1$).
See Figure~\ref{f:wn_examples}.
\end{proposition} 

\begin{proof}[Proof of part (a)] 
    Denote by $\Omega$ the given convex polygon, and by $\partial\Omega$ its boundary. 
    
    If the point $O$ lies outside $\Omega$, draw through $O$ two supporting lines to $\Omega$. 
    Take two points $A, B$ from the intersections of these supporting lines with the boundary $\partial\Omega$ of the polygon. 
    Then $w(\partial \Omega) = \frac{1}{2\pi} (\angle AOB+\angle BOA) = 0$. 
    
    If the point $O$ lies inside $\Omega$, draw an equilateral triangle $ABC$ centered at $O$.
    Take three points $A', B', C'$ where the rays $OA, OB, OC$ intersect $\partial\Omega$.
    They divide the boundary $\partial\Omega$ into three polygonal lines.
    We have
    $$w(\partial \Omega) = \frac{1}{2\pi} (\angle A'OB' + \angle B'OC' + \angle C'OA') = \frac{3}{2\pi} \angle A'OB' = \frac{3}{2\pi} \angle AOB = \pm 1.$$
\end{proof}

Part (b), depending on the exposition, is either a corollary of the \emph{Jordan theorem} \cite[\S1.4]{Sk20}, or a lemma in its proof.

Let $ABC$ be a triangle and let $O$ be a point inside it.
Then $w(ABCABC) = 2 w(ABC) = \pm 2$.
This example shows that the winding numbers of different polygonal lines with the same union of segments may be different.
Moreover, there exists a closed polygonal line $l$ such that $w(l)=0$ for every point $O\in\R^2-l$.
Examples are the polygonal lines $A$, $AB$, $ABCB$ for any points $A$, $B$, and $C$.

\begin{example}\label{p:ex} For any integer $n$ and any point $O$, there exists a closed polygonal line whose winding number around $O$ is equal to $n$.
\end{example}

\begin{proof}[Construction]
    If $n=0$, one example is the closed polygonal line consisting of a single point.
    If $n\ne 0$, take an equilateral triangle $ABC$ centered at $O$, oriented counterclockwise when $n>0$ and clockwise otherwise.
    Define a closed polygonal line $l$ by the formula $l := \underbrace{ABC\ldots ABC}_{|n| \text{ times}}$.
    By Proposition~\ref{p:winsim}.a, we have $w(l) = |n| \cdot w(ABC) = n$.
\end{proof}

\begin{proposition}\label{p:noncl} 
    The winding number $w(A_1\ldots A_m)$ is an integer. 
\end{proposition}

\begin{proof}
    For $j=1,\ldots,m$ denote $t_j:=\angle A_jOA_{j+1}$, where $A_{m+1}=A_1$.  
    Then 
    $$\overrightarrow{OA_1} \upuparrows e^{it_m}\overrightarrow{OA_m} \upuparrows e^{i(t_m+t_{m-1})}\overrightarrow{OA_{m-1}} \upuparrows \ldots \upuparrows e^{i(t_m+t_{m-1}+\cdots+t_1)}\overrightarrow{OA_1}.$$
    Therefore, $(t_m+t_{m-1}+\cdots+t_1)/2\pi$ is an integer.
\end{proof}

\emph{Hint for an alternative proof:} by Proposition~\ref{p:angles},
$$\angle A_{m-1}OA_m+\angle A_mOA_1 \equiv \angle A_{m-1}OA_1 \mod2\pi.$$
\emph{Yet another proof} of Proposition~\ref{p:noncl} is given by Propositions~\ref{p:winsim}.a and \ref{p:recur}. 

\begin{remark}[Generalizations]
(a) Let us define the \textit{winding number} $w(l,O)$ for a closed oriented \emph{continuous} curve $l$. 
The latter is defined as a continuous map $l:S^1\to\R^2-\{O\}$ from the circle to the plane minus a point.
Since the map $l$ is uniformly continuous, there exists an integer $m>0$ such that for any subdivision of the circle $S^1$ into $m$ equal arcs, and any points $x,y$ on the same arc, the points $l(x)$ and $l(y)$ are not symmetric with respect to $O$.
For $k=1,\ldots,m$ denote by $A_k$ the endpoint of the $k$-th arc.
Define
$$2 \pi \cdot w(l,O) := 
\angle A_1OA_2+\angle A_2OA_3+\ldots+\angle A_{m-1}OA_m+\angle A_mOA_1.$$
This is well-defined (see, for example, \cite[\S3]{Sk20}). 

(b) The \textit{winding number} $w(C, O)$ of an integer 1-cycle \cite[\S4]{Sk} $C$ around a point $O$ not belonging to any segment of $C$ is defined to be
$$w(C, O) = \sum_{AB \in C} \angle AOB.$$
Analogously to Proposition~\ref{p:noncl}, the number $w(C, O)$ is an integer.
Which other statements generalize to 1-cycles?
\end{remark}

\section{The simplest properties of the winding number}\label{s:prop} 
 
\begin{theorem}[Borsuk--Ulam]\label{borul-pl} 
    Let $A_1\ldots A_{2k}$ be a closed polygonal line not passing through a point $O$ that is a center of symmetry of the polygonal line (that is, $O$ is the midpoint of the segment $A_jA_{k+j}$ for each $j=1,\ldots,k$).
    Then the winding number of the polygonal line around $O$ is odd.
\end{theorem}

\begin{figure}[ht] \centering
    \includegraphics[scale=0.65]{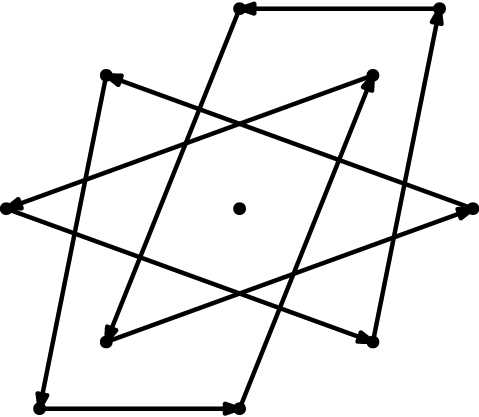}
    \caption{A closed polygonal line $l$, symmetric with respect to the point $O$; $w(l)=3$} 
\end{figure}

The following notation and results will be useful.

A \textbf{polygonal line} is an ordered collection of points (not necessarily distinct).\footnote{In mathematics, different things sometimes have the same formalization; then the difference appears in what we do with these things.
For example, many different things are formalized by the notion of a `set'.
Then formally identical objects are given different names in order to make the operations on them more understandable.
For example, in this text an ordered collection of points is called both a polygonal line and a closed polygonal line. 
Another example: in order to make multiplication more transparent, a `sequence of numbers' is called the `generating function of that sequence'.}

Let $l=A_1\ldots A_m$ be a polygonal line not passing through a point $O$. 
Define the real number $w'(l)=w'(l,O)$ by the formula
$$2\pi \cdot w'(l) = 2\pi \cdot w'(l,O):= \angle A_1OA_2+\angle A_2OA_3+\ldots+\angle A_{m-1}OA_m.$$
Clearly,

$\bullet$ $w'(A_1\ldots A_mA_1) = w(A_1\ldots A_m)$;

$\bullet$ $w'(A_1\ldots A_m) = w'(A_1\ldots A_j) + w'(A_j\ldots A_m)$ for each $j=1,\ldots,m$; 

$\bullet$ if the points $A_2,\ldots, A_{m-1}$ lie inside the angle $\angle A_1 O A_m$, then  
$2\pi w'(A_1\ldots A_m) = \angle A_1 O A_m$.

\begin{proposition}\label{p:w'}
    We have $\angle A_1OA_m = 2\pi w'(A_1\ldots A_m)+2\pi n$ for some integer $n$. 
\end{proposition}

\begin{proof}
    Clearly, $2\pi w(A_1\ldots A_m) = 2\pi w'(A_1\ldots A_m)+\angle A_mOA_1$.
    Proposition~\ref{p:noncl} implies the required statement.
\end{proof}

\begin{proof}[Proof of Theorem~\ref{borul-pl}]
    By symmetry, $w'(A_1 \ldots A_{k+1}) = w'(A_{k+1} \ldots A_{2k}A_1)$. 
    Then
    $$w(A_1\ldots A_{2k}) = w'(A_1\ldots A_{2k}A_1) = w'(A_1\ldots A_{k+1}) + w'(A_{k+1}\ldots A_{2k}A_1) =$$
    $$= 2w'(A_1 \ldots A_{k+1}) \stackrel{(1)}{=} 2 \left( \frac{\angle A_1OA_{k+1}}{2\pi} + n \right) = 1 + 2n$$
    for some integer $n$. 
    Here equality $(1)$ follows from Proposition~\ref{p:w'}.
\end{proof}

Denote by $l^{-1}$ the polygonal line obtained from a polygonal line $l$ by traversing it in the opposite direction.
Clearly, $w'(l) = -w'(l^{-1})$ for any point $O\not\in l$. 

The \emph{concatenation} of polygonal lines $l_1 = A_1 \ldots A_m C$ and $l_2 = C B_1\ldots B_k$ is the polygonal line 
$$l_1l_2 := A_1 \ldots A_m C B_1 \ldots B_k.$$
The \emph{concatenation} of closed polygonal lines $l_1 = A_1 \ldots A_m C$ and $l_2 = B_1\ldots B_k C$ is the closed polygonal line 
$$l_1l_2 := A_1 \ldots A_m C B_1 \ldots B_k C.$$
In this notation 
$$w'(l_1l_2) = w'(l_1) + w'(l_2)$$ 
for any point $O$ and polygonal lines $l_1, l_2$ not passing through $O$.
This obvious equality has already been used in the proof of Theorem~\ref{borul-pl}. 

\begin{proposition}\label{k23}
    Let $O$, $A$, $B$ be three pairwise distinct points.
    Let $l_1$, $l_2$, $l_3$ be polygonal lines joining the point $A$ to the point $B$, and not passing through the point $O$.
    Then 
    \linebreak
    $w(l_1 l_2^{-1}) + w(l_2 l_3^{-1}) = w(l_1 l_3^{-1})$.
\end{proposition}

\begin{proof}
    We have
    $$w(l_1 l_2^{-1}) + w(l_2 l_3^{-1}) = w'(l_1) + w'(l_2^{-1}) + w'(l_2) + w'(l_3^{-1}) = w'(l_1) + w'(l_3^{-1}) = w(l_1 l_3^{-1}).$$  
\end{proof}

\begin{example}[see the construction in {\cite[\S2.A]{AMS}}]\label{k23-ex}
    Let $O$, $A$, $B$ be three pairwise distinct points.
    For any three integers $n_1, n_2, n_3$ such that $n_1 + n_2 = n_3$, there exist three polygonal lines $l_1, l_2, l_3$ joining the point $A$ to the point $B$, not passing through the point $O$, such that
$$w(l_1l_2^{-1}) = n_1, \qquad w(l_2l_3^{-1}) = n_2 \quad\text{and}\quad w(l_1l_3^{-1}) = n_3.$$
\end{example}

\begin{proposition}\label{p:recur}
Let $A_1 \ldots A_m$ be a closed polygonal line not passing through the point $O$.
Let $P$ be a point such that $O$ does not lie on any of the segments $PA_1,\ldots,PA_m$.
Then
$$w(A_1 \ldots A_m) = \sum_{i = 1}^{m}w(PA_iA_{i+1}),\quad\text{where}\quad A_{m+1} = A_1.$$
\end{proposition}

\begin{proof}
    Denote $a_i := A_iA_{i+1}$ and $p_i = PA_i$ for each $i = 1, \ldots, m$.
    Using Proposition~\ref{k23} several times, we obtain
    $$w(A_1 \ldots A_m) = w(a_1\ldots a_m) =  \sum_{i = 1}^{m} w(p_{i-1}a_ip_i^{-1}) = \sum_{i = 1}^{m}w(PA_iA_{i+1}),\quad\text{where}\quad p_0 = p_m.$$
\end{proof}

\begin{proposition}\label{p:3rays} 
    Let $A_0A_1A_2$ be an equilateral triangle and let $O$ be its center.
    For $j=0,1,2$ denote by $l_j$ a polygonal line that does not intersect the ray $OA_j$ and joins $A_{j+1}$ to $A_{j+2}$, where the numbering is taken modulo $3$.
    Then $w(l_0l_1l_2) = \pm 1$.  
\end{proposition}

\begin{proof}
    We may assume that the vertices of the triangle are numbered counterclockwise.
    By Proposition~\ref{p:w'}, we have $w'(l_0) = \frac{1}{3} + k$ for some integer $k$.
    In the next paragraph, we prove that $k = 0$.
    The equalities $w'(l_1) = w'(l_2) = \frac{1}{3}$ can be proved analogously.
    Then $w(l_0l_1l_2) = w'(l_0) + w'(l_1) + w'(l_2) = 1$. 

    Denote by $B_1 \ldots B_m$ the consecutive vertices of the polygonal line $l_0$. 
    For $j = 1, \ldots, m-1$ let $t_j := \angle B_jOB_{j+1} \in (-\pi, \pi)$.
    If there were some $j= 1, \ldots, m-1$ such that $T_j := t_1 + \ldots + t_j \not \in \left(-\frac{2\pi}{3}, \frac{4\pi}{3}\right)$, then for the smallest such $j$ the segment $B_jB_{j+1}$ would intersect the ray $OA_0$. 
    Hence $T_j \in \left(-\frac{2\pi}{3}, \frac{4\pi}{3}\right)$ for every $j = 1, \ldots, m-1$.
    In particular, $T_{m-1} \in \left(-\frac{2\pi}{3}, \frac{4\pi}{3}\right)$.
    On the other hand, $T_{m-1} = 2\pi w'(l_0) = \frac{2\pi}{3} + 2\pi k$ for some integer $k$.
    Hence, $k = 0$.
\end{proof}

Other proofs are given in~\S\ref{s:wint} and in \cite[end of \S3]{AS}.

\section{The winding number and intersections}\label{s:wint} 

\begin{proposition}\label{p:rel} 
Let $P_0$ and $P_1$ be points joined by a polygonal line that does not intersect a closed polygonal line $l$. 
Then $w(l,P_0)=w(l,P_1)$. 
\end{proposition}
    
\emph{Hint:} use \emph{continuity considerations} or deduce this from Lemma~\ref{l:square}.a.

\medskip
A collection of points is said to be in \emph{general position} if no three of them lie on a line, and no three segments joining them have a common interior point.
\emph{In the rest of this section, the vertices of all polygonal lines $l$, as well as of all pairs of polygonal lines $l$ and $p$, are pairwise distinct and (unless stated otherwise) are in general position.}

\begin{proposition}\label{p:stokes} 
    Let $l$ be a closed polygonal line and $p$ a non-closed polygonal line.
    Denote by $P_0$ and $P_1$ the initial and terminal points of $p$.
    Then
    
(a) $w(l,P_1)-w(l,P_0) \equiv |l\cap p| \mod2$.

(b) $w(l, P_1) \equiv |l\cap p| \mod2$, if the point $P_0$ is sufficiently far from $l$ (for example, is outside the convex hull of $l$).
\end{proposition}

\begin{proof}[Sketch of proof]
Part (b) follows from part (a), since $w(l, P_0) = 0$.

(a) If $p$ does not intersect $l$, then Proposition~\ref{p:rel} gives $w(l,P_1)-w(l, P_0) = 0 = |l \cap p|$.
     
Otherwise, we may assume that $p$ is a segment that intersects $l = A_1 A_2 \ldots A_m$ at only one point, and this point belongs to the segment $A_1A_2$.
Choose a point $X$ such that $P_0$ lies inside the triangle $A_1XA_2$.  
Then $P_1$ lies outside the triangle $A_1XA_2$.
Define a closed polygonal line $l'$ by the formula $l' := A_1 X A_2 \ldots A_m$.
We have 
$$w(l, P_1) = w(l, P_1) + w(A_1XA_2, P_1) = w(l', P_1) \stackrel{(1)}{=} 
w(l', P_0) = w(l, P_0) + w(A_1XA_2, P_0) = w(l, P_0) \pm1.$$
Here equality (1) follows from Proposition~\ref{p:rel}.

\emph{Another approach:} one can apply Lemma~\ref{l:square}.b (or even its modulo $2$ version). 
\end{proof}

\begin{figure}[ht]
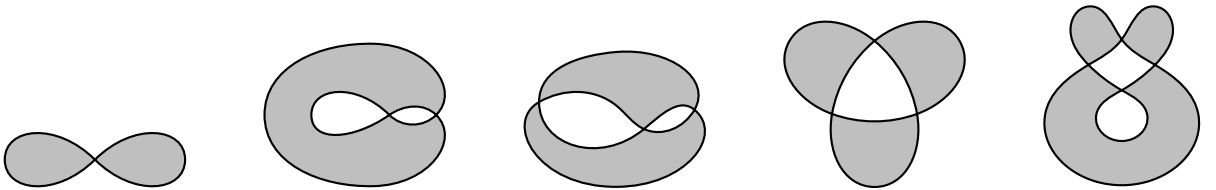
\centering
    \includegraphics[scale=.8]{c1.eps} 
    \qquad 
    \includegraphics[scale=.8]{c2.eps}
    \qquad
    \includegraphics[scale=.8]{c3.eps}
    \qquad 
    \includegraphics[scale=.8]{c4.eps}
    \qquad 
    \includegraphics[scale=.8]{c5.eps}
    \caption{Checkerboard colorings of complements and the interior mod $2$}
    \label{f:tohu}
\end{figure}

Any (open) subset of the plane splits into \emph{connected components} such that any two points in the same component can be joined by a polygonal line lying in the subset, whereas no two points from different components can.
The complement $\R^2-l$ of a closed polygonal line $l$ admits a `checkerboard coloring' such that the components adjacent along some segment of $l$ are colored differently, see Figure~\ref{f:tohu} \cite[Proposition 2.1.1]{Sk18}.  
The \textit{interior modulo $2$} of $l$ is the union of the black regions in the checkerboard coloring (assuming that the `infinite' region is white).
Here is an equivalent direct definition not using checkerboard coloring. 
The \textit{interior modulo $2$} is the set of all points $X\in\R^2-l$ for which there exists a polygonal line $p$, 

$\bullet$ joining the point $X$ to a point sufficiently far from $l$, 

$\bullet$ intersecting $l$ in an odd number of points, and 

$\bullet$ such that all vertices of $l$ and $p$ are pairwise distinct and (except for $X$) are in general position. 

This is well-defined by the parity lemma \cite[Lemma 1.3.2.b]{Sk18}: \emph{if the vertices of two closed polygonal lines are pairwise distinct and in general position, then the polygonal lines intersect at an even number of points}. 
(This lemma is a special case of Proposition~\ref{p:stokes}.a for $P_0 = P_1$.)

\begin{proposition}[follows from Proposition~\ref{p:stokes}.b]\label{p:radon}
A point $O$ lies in the interior modulo $2$ of a closed polygonal line $l$ if and only if $w(l)$ is odd.
\end{proposition}

\begin{figure}[h]\centering
\includegraphics[scale=.8]{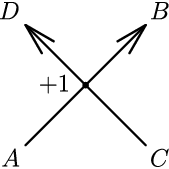}\quad \includegraphics[scale=.8]{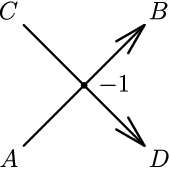}
\caption{The sign of an intersection point}
\label{f:sign}
\end{figure}

\begin{figure}[h]
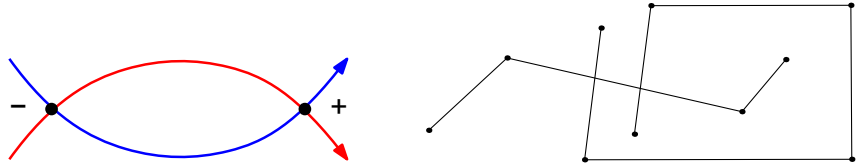

\centerline{\includegraphics[width=4.5cm]{intersection_sign_graphs.eps} \qquad \includegraphics{curves.eps}}
\caption{Two polygonal lines in the plane intersecting in an even number of points, whose sum of signs is zero (left) and non-zero (right)} 
\label{f:gl2}
\end{figure}

Assume that no three of the points $A,B,C,D$ lie on one line.
The {\bf sign} of the intersection point of the oriented segments $\overrightarrow{AB}$ and $\overrightarrow{CD}$ is defined to be $+1$ if the traversal $ABC$ is clockwise, and $-1$ otherwise (Figure~\ref{f:sign}).
For polygonal lines $l$ and $p$, denote by $l \cdot p$ the sum of the signs of their intersection points (Figure~\ref{f:gl2}).

\begin{proposition}\label{p:sign-stokes} 
    Let $l$ be a closed polygonal line and $p$ a non-closed polygonal line.
    Denote by $P_0$ and $P_1$ the initial and terminal points of $p$.
    Then
    
(a) $w(l,P_1)-w(l,P_0) = l \cdot p$.  

(b) $w(l,P_1) = l \cdot p$, if the point $P_0$ is sufficiently far from $l$. 
\end{proposition}

Proposition~\ref{p:sign-stokes} is proved analogously to Proposition~\ref{p:stokes} or follows from Lemma~\ref{l:square}.b. 

\begin{figure}[ht]
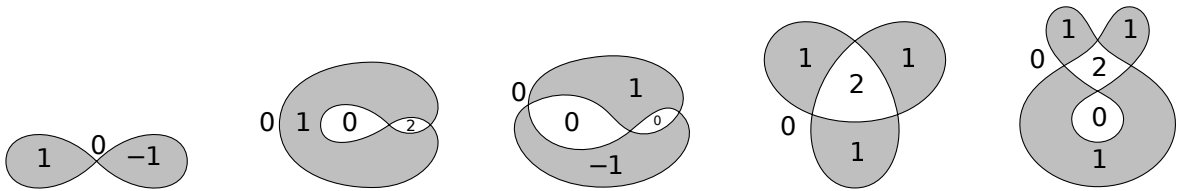
\centering
    \includegraphics[scale=.8]{numbered-curve2.eps}
    \qquad
    \includegraphics[scale=.8]{numbered-curve3.eps}
    \qquad
    \includegraphics[scale=.8]{numbered-curve4.eps}
    \qquad
    \includegraphics[scale=.8]{numbered-curve5.eps}
    \qquad
    \includegraphics[scale=.8]{numbered-curve1.eps}
    \caption{`Colorings' of complements according to the winding number (up to reversing the direction of the polygonal line)}
    \label{f:tohu_colored}
\end{figure}

Consider a closed polygonal line $l$ (Figure~\ref{f:tohu_colored}).
By Proposition~\ref{p:rel}, the winding numbers around points in the same component of the complement $\R^2-l$ are equal. 
By Proposition~\ref{p:sign-stokes}.a, the winding numbers in components adjacent along some segment differ by $\pm 1$.
Such a `coloring' of the complement $\R^2-l$ by integers is called the \emph{M\"obius--Alexander coloring}. 

For polygonal lines $A\ldots B$ and $C\ldots D$, define
$$\partial(A\ldots B \times C\ldots D) := w'(A\ldots B,D)-w'(C\ldots D,B)-w'(A\ldots B,C)+w'(C\ldots D,A).$$ 
This number is defined when no endpoint of either polygonal line lies on the union of the edges of the other polygonal line. 
(Represent the polygonal lines as \emph{piecewise-linear maps} $l,p:[0,1]\to\R^2$.
Then $\partial(l\times p)$ is the winding number of the vector $l(x)-p(y)$ as the point $(x,y)$ traverses the boundary $\partial([0,1]^2)$ of the square $[0,1]^2$ counterclockwise.)

For the polygonal lines $l,p$ in Figure~\ref{f:gl2}, left and right, respectively, $\partial(l\times p)$ equals $0$ and $\pm2$ (check this!). 
 
\begin{lemma}\label{l:square} 
(a) For any disjoint polygonal lines $l$ and $p$, we have $\partial(l\times p)=0$. 

(b) For any polygonal lines $l$ and $p$, we have $\partial(l\times p) = l\cdot p$.  
\end{lemma}

\begin{proof}[Sketch of proof]
(a) If $l=AB$ and $p=CD$, then the lemma holds because 
\linebreak
$\angle ADB + \angle DBC + \angle BCA + \angle CAD = 0$ for any disjoint segments $AB$ and $CD$.
(It is enough that this sum of angles lies in $(-2\pi,2\pi)$.)

Reduce the general case to this one.     
For arbitrary polygonal lines $l_1, l_2, p$ we have 
    $$\partial(l_1l_2 \times p) = \partial(l_1 \times p) + \partial(l_2 \times p)\quad\text{and} \quad\partial(p \times l_1l_2) = \partial(p \times l_1) + \partial(p \times l_2).$$
    Therefore, 
    $$\partial(l \times p) = \sum_{\begin{array}{l}
         EF  \text{ is a segment of }  
         l  \\
         GH  \text{ is a segment of }
         p    
         \end{array}} \partial(EF\times GH)=0.$$

(b) Analogously to part (a). 
\end{proof}

\begin{proof}[Proof of Proposition~\ref{p:3rays} using Proposition~\ref{p:rel}] 
(Obtained by editing the text of A. Abzalilov.)
Define a polygonal line $a_j := A_{j+1}A_{j+2}$ and a closed polygonal line $m_j := l_ja_j^{-1}$.
Then 
$$w(l_0l_1l_2) \stackrel{(1)}{=} w(m_0) + w(m_1) + w(m_2) + w(a_0a_1a_2) \stackrel{(2)}{=} w(A_0A_1A_2) = \pm1,\quad\text{where}$$

$\bullet$ equality (1) is obtained using Proposition~\ref{k23} several times;

$\bullet$ equality (2) follows from the equalities $w(m_0) = w(m_1) = w(m_2) = 0$ proved in the next paragraph.

Take any $j = 0, 1, 2$.
On the ray $OA_j$ there exists a point $P_j$ outside the convex hull of the points of the polygonal line $m_j$.
Then $w(m_j, P_j) = 0$.
The segment $OP_j$ does not intersect the polygonal line $m_j$.
Therefore, $w(m_j) = 0$ by Proposition~\ref{p:rel}.
\end{proof}